\documentclass{amsart}
\usepackage{amsthm,amsmath}
\usepackage{amsaddr}

\usepackage{dsfont,pstricks,pst-node}

\def\N{\mathds{N}}

\def\P{\mathds{P}}
\def\R{\mathds{R}}

 \def\X{\mathcal{X}}
 \def\Y{\mathcal{Y}}
\def\Z{\mathds{Z}}

\newcommand{\indicator}[1]{\mathds{1}_{\{#1\}}}

\theoremstyle{plain}
\newtheorem{thm}{Theorem}[section]
\newtheorem{prop}{Proposition}[section]

\newtheorem{corollary}{Corollary}[section]

\theoremstyle{definition}

\newtheorem{remark}{Remark}[section]

\begin{document}

\title{Metastability in communication networks}

\author[D.\ Tibi]{\vspace{-0.5cm}Danielle Tibi}
\address{\vspace{-0.2cm} LPMA, 
Universit\'e Paris 7,  \\case 7012, 175 rue du Chevaleret, 75013 Paris, France}
\email{danielle.tibi@math.jussieu.fr}

%\begin{keyword}
%\kwd{Mean field limit}
%\kwd{Lyapunov function}
% \kwd{metastability}
 %\kwd{large deviations}
 %\kwd{quasipotential}
 %\end{keyword}

%\subjclass[2000]{??}

\date{\today}

\bibliographystyle{plain}

\begin{abstract} 
Two  models of  loss networks,  introduced by Gibbens \emph{et al.} ~\cite{Gibbens:90} and  by Antunes \emph{et al.} ~\cite{Antunes:08},  are known to exhibit a mean field limiting regime with several stable equilibria. 

These models  are reexamined in the light of Freidlin and Wentzell's large deviation approach of randomly perturbed dynamical systems. Assuming that some of their results still hold under slightly relaxed conditions, the metastability property is derived for both systems.

A Lyapunov function  exhibited in  ~\cite{Antunes:08} is next identified with the quasipotential associated with a slightly modified, asymptotically reversible,  Markovian perturbation of the same dynamical system.

Another interpretation,  in terms of entropy dissipation, of the Lyapunov function in   ~\cite{Antunes:08} is finally given.  The argument  extends to another, similar but closed model.

\end{abstract}

\maketitle

\section{Introduction}

  Metastability has given rise to a profuse literature  in the Statistical Physics  and Probability Theory areas during the last decades. It concerns a wide range of models, from  the  Curie-Weiss model to lattice  gas and spin    models   and to    diffusion processes. See ~\cite{Olivieri:05}  and ~\cite{Bovier:06}  for an overview of the subject. Metastability can be roughly described as the phenomenon  occuring when a physical system stays a very long time  in some abnormal state before reaching  its normal - under the prevailing conditions - equilibrium state. The normal  situation is only restored after, under some random perturbation or some other external provision of energy, the system can get over some energy barrier. Mathematically, it is usually formalized through exponential growth of some  exit times  under some asymptotic, supposed to correspond to the physics of the system. 
  
  While the potential theoretic approach, more recently developed (see~\cite{Bovier:06}), provides sharp estimations on crossover times for  reversible dynamics,  a lot is still due to the large deviation approach developed by Freidlin and Wentzell in ~\cite {Freidlin:84} (see  ~\cite{Olivieri:05}  for an outline of  the application range). It has lead to some quite complete descriptions of the  \emph{exit path}  from a metastable state, e.g.~for the Ising model under Glauber dynamics at low temperature.  Yet again, reversibility is often crucial,  even in this context.
  
   Metastability  is  expected to occur for some  specific models in communication networks, but a formal proof is still lacking.  Namely, ~\cite{Gibbens:90} and ~\cite{Antunes:08} analyze two loss systems with local interactions, that admit a mean field limit as the number of nodes goes to infinity. The  model in  ~\cite{Gibbens:90} is  ruled  by an alternative routing  procedure for blocked calls, adapted from  ~\cite{Marbukh:85}, while   ~\cite{Antunes:08} analyzes a simple model for a loss network with mobile customers.  In both papers, the limiting dynamics are shown to exhibit a phase with several stable equilibrium points.  ~\cite{Gibbens:90} provides estimations of certain exit times for a one dimensional diffusion approximation of the model, suggesting that metastability occurs. This should also be  the case for the model in ~\cite{Antunes:08}, as suggested by simulations. 
   
  Yet, much less is known for these network models than for the  classical examples  cited above. In particular, their Markovian dynamics are  not reversible and computing the invariant distributions is out of reach. Such useful quantities as Hamiltonian, energy barrier etc... are thus not available. However, it is  expected that as the number $N$ of nodes grows, the invariant distribution  should be approximately given by  $ Z_N^{-1} \exp (- N  h)$ for some energy function $h$ that would play the role of the Hamiltonian in the limit.  In   ~\cite{Antunes:08} a Lyapunov function for the limiting dynamical system is exhibited (and used for proving multistability), but this function has no reason to describe the correct energy landscape. In particular, it is not known which equilibrium points are asymptotically relevant in the invariant distribution (i.e., correspond to  global minima of the possible Hamiltonian).
  \\
  
   The main purpose of the present paper is to show that the models in ~\cite{Gibbens:90} and ~\cite{Antunes:08} essentially  fit the scheme of Freidlin and Wentzell (\cite {Freidlin:84}), as - in these authors' terminology -    \emph{locally infinitely divisible processes}.  This is the object of Section ~\ref{sec:quasipot}. As a result, exponential growth   (as $N$  gets large) of exit times from neighborhoods of stable equilibrium points,  that is,  metastability is obtained.    The location of the exit points can also  be described. 

 It must be  pointed out that the large deviations results  stated  in ~\cite {Freidlin:84}   are not rigorously applicable  to the processes of interest in this paper, since these evolve in some compact subset of  $\R^d$, on the frontier of which some of the technical hypothesis required in ~\cite {Freidlin:84}  are no longer valid.  Yet,  it  seems  that there is  no fundamental reason for these restrictions. Extension of the  results  in ~\cite {Freidlin:84} to this slightly more general  context will thus  be  used without proof. Such a   proof
 is  beyond the scope of this paper, that  aims at opening a way for understanding  the  stochastic behavior of  systems for which only the deterministic limiting evolution has been described so far.

  A second issue of this paper is to decrypt  the Lyapunov function  exhibited in ~\cite{Antunes:08}. Two answers are given in this direction. One, presented in Section ~\ref{sec:quasipot}, is related to  the  \emph{quasipotential}  introduced in ~\cite {Freidlin:84}, which is  the crucial quantity involved in  estimations of exit times and  exit points. Moreover, as suggested in ~\cite {Freidlin:84} from  the thoroughly analyzed case of diffusion processes, when there is just one equilibrium point, the quasipotential   should  represent the underlying energy function associated with the stationary state.  It is here   proved that the Lyapunov function  exhibited in ~\cite{Antunes:08}  coincides with the quasipotential  of a slightly modified version of the process of interest. In addition, a heuristic argument  suggests  that  the modified process is \emph{asymptotically reversible} as the number of nodes gets large.

 Section ~\ref{sec:lyapunov} is   devoted to   another interpretation of this Lyapunov function,  related to the well known decrease  of   relative entropy  along a semi-group. This relies on a very particular feature of the model in ~\cite{Antunes:08}. Yet, two  other  models  (particular in some other sense) can then be introduced, to which this principle for  obtaining  a Lyapunov  function can  be exported. This  helps proving convergence of their invariant distribution to a Dirac mass.

    Section ~\ref{sec:mean_field} recalls  the two  models  of interest  and  main  results from  ~\cite{Gibbens:90} and ~\cite{Antunes:08}.

\section{Mean field limits, multistability} \label{sec:mean_field}
This section gives a short review of the models  in  ~\cite{Gibbens:90}  and ~\cite{Antunes:08} and  of the main results therein. These models exhibit a mean field limiting dynamics of some generic form, of which two other examples will be given in Section~\ref{sec:lyapunov}. But only the two models from  ~\cite{Gibbens:90}  and ~\cite{Antunes:08} exhibit a phase with several stable points.

\subsection{Multistability due to rerouting}\label{subsec:GHK}

A first example of bistability in queueing networks is given in    ~\cite{Gibbens:90}.    The analysis is formalized through   convergence of some family of empirical processes  to some dynamical system. All the subsequent examples of this section and Section~\ref{sec:lyapunov}  will fit this  frame. 

 This first  model is a simplified version  of a network with  alternative routing  proposed in ~\cite{Marbukh:85}.  In ~\cite{Anantharam:91}, a lattice version,  with long range rerouting, is proposed.

The model in ~\cite{Gibbens:90} is the following: The  network  considered consists of  $N$ nodes offering the  same finite (integer) capacity $C \ge 1$. Customers enter the network at the different nodes according to $N$ independent  Poisson processes with intensity $\lambda >0$. When some customer arrives at some node where no more than $C-1$ customers are present, he  occupies one unit of capacity for an exponentially distributed service time with mean one. When some customer arrives at some saturated node, he is rerouted to two other nodes, chosen uniformly among the $N-1$ possible nodes. If both chosen nodes have one unit of capacity available, the customer then behaves like two independent customers, leaving the nodes after two independent exponential  times with mean one.  In the contrary (i.e., if at least one of the two nodes is saturated), the customer is definitively rejected from the system.

Due to symmetry with respect to the $N$ nodes, the quantity of interest is the empirical distribution of the nodes as function of time, that is
\begin{align} \label{eq:Y} Y^N(t)=\frac{1}{N} \sum _{i=1} ^N \delta _{X^N_i(t) },\end{align}
where $X_i^N(t)$ denotes the number of customers present at node $i$ at time $t$ and $\delta _x$ is the Dirac mass at $x$.

The  $X^N_i(t)$ ($N \ge 1$, $t\ge 0$  and $i=1, \dots ,N$) evolve in the set $\X = \{ 0, \dots ,C\}$. Identifying the set $\mathcal P(\X)$ of probability measures on $\X$ with the set
\[ \Y= \left \{ y=(y_0, \dots , y_C) \in [0,+ \infty[^{C+1}: \sum _{n=0} ^C y_n=1\right \},\]
 where $y_n$  represents the mass at $n$  ($n=0, \dots ,C$)  of the probability measure $y$,
one can write $Y^N(t)= (Y^N_n(t), 0 \le n \le C)$, with
\begin{align} \label{eq:Y_n} Y^N_n(t)=\frac{1}{N} \sum _{i=1} ^N \indicator {X^N_i(t) =n }.\end{align}   
In other words, $Y^N_n(t)$ is the proportion of  nodes that are in state $n$ at time $t$.\\

For any fixed $N\ge 3$, $(Y^N(t), t\ge 0)$ is a  Markov jump process with the following transitions and rates, where $e_n$ denotes the  $n^{th}$ unit vector in $\R^{C+1}$:

\[ y \longrightarrow \left \{ \begin{array}{l}  y+\frac{1}{N} (e_{n+1}-e_n) \\ \\
 y+\frac{1}{N} (e_{n-1}-e_n)\\
 \\ y+\frac{1}{N} \big( e_{m+1}-e_m  +e_{n+1}-e_n\big) \\
  \\y+\frac{2}{N} \big(e_{n+1}-e_n\big)
\end{array} \right.
\text{at rates}  \begin{array}{ll}
 \lambda  N y_n  & (0 \le n \le C-1) \\
\\ Nn y_n & (0 < n \le C)\\
\\ 2  \lambda  \frac{   N^3 y_C y_ny_m}{(N-1)(N-2)}  & (0 \le m \neq n \le C-1)\\ 
\\ \lambda   \frac{N^2  y_C y_n (Ny_n-1)}{(N-1)(N-2)} &(0 \le n \le C-1)

\end{array} \]

The first jump corresponds to some arrival  at some node with $n$ customers, the second one to some departure from some node with $n$ customers and the two last jumps  correspond to rerouting of some customer to two nodes with, respectively,  different or equal numbers of customers. \\

~\cite{Gibbens:90} proves that  for any $T>0$, the process $(Y^N(t), 0 \le t \le T)$ converges  in distribution, as $N$ goes to infinity, to the solution  with initial value $y(0)$ of the following  differential system of equations: for $n=0, \dots ,C$
\begin{multline*} y_n'(t)=\lambda \Big( 1+2y_C(t)(1-y_C(t)\Big) y_{n-1}(t) \mathds 1_ {n\ge 1} + (n+1) y_{n+1} (t) \mathds 1_ {n\le C-1}
 \\ - \Big(\lambda (1+2y_C(t)(1-y_C(t)) +n \Big) y_n(t).
\end{multline*}
 provided that $Y^N(0)$ converges in distribution to $y(0)$.

  The vector field characterizing the limiting dynamical system, given by the right hand sides of this differential system, can be heuristically  obtained by computing  the infinitesimal  mean jump  from position $y$ (suming up the  jump amplitudes multiplied by the corresponding rates) and letting $N$ grow to infinity.
  
 Introducing the family of infinitesimal generators $(L_y, y\in \Y)$ defined by
 \[ L_yf(n)=  \lambda (1+2 y_C(1-y_C))(f(n+1) -f(n)) \mathds 1_ {n\le C-1} + n(f(n-1)-f(n)) \quad (f \in \R ^{ \X}, n \in \X),\]
  the above differential
  system  rewrites as one unique differential equation on $\Y$:
\begin{align} \label{eq:S} \dot y = y L_y,
\end{align}
where the second member is the product of probability measure (or row vector) $y$   by the infinitesimal generator  (or rate matrix) $L_y$.  For $y\in \Y$,  $L_y$ is the generator of an $M/M/C/C$ queue with arrival rate $  \lambda (1+2 y_C(1-y_C))$ and service rate $1$.

 Equation ~\eqref{eq:S} precisely conveys the mean field property of the model: it tells that, in the limit $N \to \infty$, the empirical distribution $y(t)$ of the nodes evolves  in time as the distribution of some non-homogeneous Markov process on $\X=\{0, \dots ,C\}$, whose jump rates at time $t$ are given by $L_{y(t)}$, being  hence constantly updated according to  the current distribution, or ``mean field'' $y(t)$. These jump rates are those of an $M/M/C/C$ queue. Only the arrival rate  is  time dependent, and  more precisely  depends on  $y_C(t)$, that is, on the proportion of saturated nodes. This $M/M/C/C$ queue can be viewed as  representing the instantaneous evolution of a ``typical node'' under the global influence of the other nodes.  Due to symmetry,  for large $N$, this virtual node  summarizes  the  whole network.

All the forecoming examples (Sections~\ref{sec:mean_field} and~\ref{sec:lyapunov}) will follow this  scheme, with different types of dependency  on $y$ of the mean field arrival rate.
\\

The equilibrium points of the limiting dynamical system ~\eqref{eq:S}  are the solutions of $y L_y=0$. Since for all $y$, $L_y$ is the generator of some  ergodic Markov process on $\X$,  this means that $y$ is equal to  the unique invariant distribution associated to $L_y$.

The $M/M/C/C$ queue with arrival rate $\rho$ and service rate $1$  is known to be reversible,  having invariant distribution given by the well known Erlang distribution:
\begin{align} \label{eq:erlang} \nu _{\rho}(n)=\frac{1}{Z(\rho)}\frac { \rho ^n}{n!} \quad (n=0, \dots , C),
\end{align}
where $Z(\rho)=\sum _{n=0}^C \rho ^n/n!$ is a normalizing constant.  Equilibrium points are thus given by the solutions of the  fixed point equation
\begin{align}\label{eq:fix} y= \nu _{\rho(y)},
\end{align}
where $\rho (y)= \lambda (1+2 y_C(1-y_C))$.

~\cite{Gibbens:90} shows numerically that for certain values of $\lambda$, this equation  exhibits several solutions, namely three, among which two are stable  and one is unstable.
This suggests that for $N$ large and for  suitable values of $\lambda$, the system should be attracted to one  of two possible  states, both of the Erlang form but with two different  values of $\rho$.  Intuitively, the system  can either fall into  a heavy loaded regime  or into  a light loaded one. In the first one, the heavy load  maintains itself by inducing many  reroutings, while in the opposite way, a small proportion of saturated nodes maintains a low rate of rerouting.

\begin{remark}
A slight variant of this model consists in rerouting customers to only one other node, but instead, changing their service rate from $1$ to some value $\mu <1$. One can prove that multistability still occurs for well chosen values of $\mu$ and $\lambda$. Notice that, in order to preserve the Markov property, it is  here necessary to introduce two types of customers: those with service rate $1$ and the rerouted ones, with service rate $\mu$.  

\end{remark}

\subsection{Multistability due to coexistence} \label{subsec:AFRT} 
A second example of multistability in the networks context is given by ~\cite{Antunes:08}.  One major difference with the previous model is that here multistability can  occur only when different classes of customers coexist having different capacity requirements. Besides that, the model  deals with  mobile customers travelling from one node to another  (and being possibly rejected during their service).

Here again, the network consists of $N$ nodes with capacity $C$,  now offered to $K$ different  classes of customers. For  $k=1, \dots ,K$, each customer of class $k$ occupies $A_k$ units of capacity at each node he visits ($1 \le A_k \le C$). Class $k$ customers arrive at each  node  according to some Poisson process with intensity $\lambda_k$, and have service times  exponentially distributed with parameter $\mu _k$. When some class $k$ customer arrives at some node where his capacity requirement is not available, he is definitively rejected from the network. Otherwise, he begins to be served at this node, and then moves at rate $\gamma _k$ during service. At each move, a new node is chosen uniformly among the $N-1$ possible nodes. Customers  either leave the system through rejection at some node along their route, or through end of service. 

All arrival processes, service durations and sojourn times of customers at the different nodes are assumed independent.

The empirical distribution $Y^N(t)$  of the nodes at time $t$ is still given by equation ~\eqref{eq:Y}, but here the state   $X_i^N(t)$ of node $i$ at time $t$  is $K$-dimensional, consisting of the different numbers of customers of each class present at node $i$ at time $t$. The state spaces  $\X$ of variables $X^N_i(t), 1 \le i \le N, t\ge 0,$ and  $\Y$ of $Y^N(t), t\ge 0,$ are now
\[  {\mathcal X} = \{ n\in \N^K : \sum _{k=1}^K n_k A_k \leq C \} \]
 \[ \text{ and } \quad \Y= \mathcal P (\X)= \left \{ y=(y_n,n \in \X) \in [0,+ \infty[^{\X}: \sum _{n\in \X}  y_n=1\right \}.\]
Here again  $Y^N(t)=(Y^N_n(t), n\in \X)$ for $t \ge 0$, where $Y^N_n(t)$ is  given by equation ~\eqref{eq:Y_n}.
Note that  the process $Y^N$ actually evolves in some finite subset  $\Y ^N$ of $ \Y$:
\[ \Y ^N =   \{y=(y_n,n \in \X) \in \Y: Ny_n \in \N \text{ for } n\in \X \} .\]

 $Y^N$ is a Markov jump process with   the following transitions, where  $f_k$ denotes the $k^{th}$ unit vector in $\R ^K$ ($k=1, \dots ,K$) and  $e_n$ is the $n^{th}$ unit vector in $\R^{\X}$  ($n \in \X$):
for $1 \leq k\leq K$, $n,m \in \X$,
\[ y \longrightarrow \left \{ \begin{array}{l}  y+\frac{1}{N} (e_{n+f_k}-e_n) \\ \\
 y+\frac{1}{N} (e_{n-f_k}-e_n)\\
 \\ y+\frac{1}{N}  \big( (e_{m+f_k}-e_m)  {\mathds 1}_{m+f_k \in {\mathcal X}} +e_{n-f_k}-e_n\big) 
 \end{array} \right.
\text{at rates}  \begin{array}{l}
 \lambda _k N y_n {\mathds 1}_{n+f_k \in {\mathcal X}}  \\
\\ \mu _k Nn_k y_n \\
\\  \frac{\gamma _k N  n_k y_n}{N-1}(Ny_m-{\mathds 1} _{m=n})
\end{array} \]
The first jump corresponds to the arrival of some class $k$ customer at some node in state $n$, the second one to the
end of service  of some class $k$ customer at some node in state $n$, and the last one to a move of some class $k$ customer from some node in state $n$ to  some node in state $m$ (possibly saturated, implying rejection).  \\
 
It is proved in ~\cite{Antunes:08} that for $T>0$, the process  $ (Y^N(t), 0 \le t \le T)$ converges in distribution  to  $ (y(t), 0 \le t \le T)$
 solving the differential system of equations in $\Y$: 
 \[ y'_n(t)=\sum _{k=1}^K \Big[ \Big( \lambda _k+\gamma _k [ I_k,y(t)] \Big)  y_{n-f_k}(t) {\bf 1}_{n_k\geq 1} +(\mu _k +\gamma _k)(n_k+1) y_{n+f_k}(t) {\bf 1}_{n+f_k\in {\mathcal X}}\]
  \[-\Big( (\lambda _k +\gamma _k  [ I_k,y(t) ])  {\bf 1}_{n+f_k\in {\mathcal X}}+(\mu_k +\gamma _k)n_k\Big) y_n(t) \Big] \quad (n\in \X),\]
 if $Y^N(0)$ converges in distribution to  $y(0)$.
  
 Here $\displaystyle{ [ I_k,y ]= \sum _{n\in {\mathcal X}} n_k y_n}$  is the  mean of  the $k^{th}$ marginal of $y$ (in particular $ [ I_k,Y^N(t) ]$  is the   number  of class $k$ customers per node  at time $t$ in the network, that is, the density of class $k$ customers present). Convergence in distribution refers to the Skorohod topology on $\mathcal D([0,T])$.\\

This system can  here again be written as ~\eqref{eq:S} where  $L_y$ is now the infinitesimal generator of an $M/M/C/C$ queue with $K$ classes of customers having different arrival rates $ \lambda _k+\gamma _k [ I_k,y]$, service rates  $\mu _k +\gamma _k$ and capacity requirements $A_k$ ($k=1, \dots ,K$).

Note that  in ~\cite{Dawson:05}  a similar mean field limiting dynamics, described by a non linear equation in the form of \eqref{eq:S}, is obtained for a system of $N$ interacting queues. The  $L_y$ involved are birth and death  generators with parameters depending  on the mean of $y$, as in our present case.  However, a major difference with the  present model is  that the stochastic dynamics  itself involves a mean field interaction:  The arrival rates of customers in the network depend on the global density of occupation. On the contrary, in the  model from  ~\cite{Antunes:08},  interaction between customers is \emph{local} (only due to saturation at some node) and the mean field evolution only appears in the limit. 
 \\

Erlang formula ~\eqref{eq:erlang} for the invariant distribution of  the $M/M/C/C$ queue extends to the case of  $K$ classes of customers with capacity requirements $A_k$ and arrival rate-to-service rate ratios $\rho _k$ ($k=1, \dots ,K$), where $n!$ and $\rho ^n$  ($n \in \X$)  now hold for
\[n!=\prod _{k=1}^K n_k! \quad \text{ and } \rho ^n= \prod _{k=1}^K \rho _k ^{n_k} \quad (n=(n_1, \dots ,n_K) \in \X).\]
  $Z(\rho)$ is here given by $\displaystyle{ Z(\rho)= \sum _{n \in \X} \frac{\rho ^n}{n!}}$. Equilibrium points are then again characterized by ~\eqref{eq:fix}, where $\rho(y)$ is now multidimensional: $\displaystyle{\rho(y)= \left  (\frac{ \lambda _k+\gamma _k [ I_k,y] }{\mu _k + \gamma _k}\right )_{ 1 \le k \le K }}$.\\ 

In ~\cite{Antunes:08}, coexistence of several equilibrium points is  proved to occur when $K=2$, $A_1=1$ and $A_2=C$, for $C$ sufficiently large and  for certain values of parameters $\lambda _k, \mu _k$ and $\gamma _k$ ($k=1,2$). A key argument  is the determination of a Lyapunov function for dynamical system ~\eqref{eq:S}, that is, some continuously  differentiable, bounded from below, function $g$  defined on $[0,  + \infty[^{\X}$  such that 
\[ y L_y \nabla g(y)  \le 0 \quad (y \in \Y \subset [0,  + \infty[^{\X}),\]
where equality holds only if $yL_y=0$, i.e., if $y$ is an equilibrium point of the dynamics.

$g$ is explicitely given by
\begin{align} \label{eq:g} g(y)= \sum _{n \in \X} y_n \log(n!y_n) - \sum _{k=1} ^ K \int _0^{ [ I_k,y ]} \log \frac{ \lambda _k+\gamma _k x}{\mu _k +\gamma _k} dx.\end{align}
Moreover, $g$ satisfies: for $y \in \Y$, 
 \[  yL_y=0 \Longleftrightarrow \nabla g(y)\perp \Y ,\]
 so that equilibrium points are characterized as the critical points of $g_{|\Y}$. An analytic function argument shows that these critical points are isolated.

These properties of $g$ allow one to discriminate  \emph{stable} (local  minima of $g$) from  \emph{unstable} (local maxima and saddle points of $g$) equilibrium points.

Multistability is then proved (for certain values of the  parameters) by proving existence of a saddle point for $g$, and then showing that two local minima are necessarily present, one on each side of some line crossing the saddle point.\\

Besides the  multistability issue, a Lyapunov function is a tool for showing that  equilibrium points of the limiting  dynamical system are the concentration points, as $N \to \infty$, of the invariant measures $\pi^N$ of processes $Y^N$ (note that $Y^N$ is an irreducible  Markov jump process on  the finite state space $\Y^N$, so that it admits a unique invariant measure denoted $\pi ^N$). This means in some sense commutation of limits as $N \to \infty$ and  $t\to + \infty$. 

 More precisely, it is proved in ~\cite{Antunes:08} that the infinitesimal generator  $\Omega ^N$ of $Y^N$ converges to the generator  $\Omega $ of  the limiting (degenerated) Markov process given by ~\eqref{eq:S}. $\Omega$  is defined, at any $C^1$ function $f$ on $\R^{\X}$, by
\[\Omega f(y)=  yL_y \nabla f(y).\]
 It is then  standard  that any weak limit of the sequence $(\pi ^N)$ is an invariant measure for $\Omega$.
The Lyapunov function then makes it possible to show   that the  invariant measures of $\Omega$, hence the weak limits of $(\pi ^N)$, are precisely the convex combinations of Dirac masses at equilibrium points of   ~\eqref{eq:S}.  In particular, when the  equilibrium point $\bar y$ is unique,  $\pi^N$  converges to the Dirac mass  at $\bar y$.
We refer to ~\cite{Antunes:08}  for   details, or to the proof of Proposition ~\ref{prop:closed} hereafter.

\section{Large deviations, quasipotential} \label{sec:quasipot}

As just recalled,  the Lyapunov function $g$ associated with the model in ~\cite{Antunes:08} helps describing the weak limits of the stationary distributions $\pi ^N$ of  the corresponding processes $Y^N$. Such a function is not available  in the  case of ~\cite{Gibbens:90}. Indeed, Sections ~\ref{sec:quasipot} and ~\ref{sec:lyapunov} will emphasize the specificity of the model of  ~\cite{Antunes:08} that makes the explicit formulation  of $g$ possible.

 Note however that the previous analysis of the model in ~\cite{Antunes:08} does not tell which stable equilibria remain significant in the limit, in the multistable phase. It is not even proved that  the weak limits of $\pi ^N$  have positive mass only  at  \emph{stable} equilibrium points.\\

In this respect, a more precise issue  would be to find, for both models of interest, an {\em energy function}  $h$ describing  their invariant distribution $\pi ^N$ as $N \to \infty$ in the sense that 
\[ \pi ^N(y) \approx \frac {e^{-N h(y)}}{Z_N} \quad \mathrm { for } \  y \in \mathcal Y ^N.\] 
Such an $h$ is expected to be a Lyapunov  function for the limiting dynamical system, or at least to satisfy $  yL_y \nabla h(y)  \le 0$ (see Remark ~\ref{remark:balance} below).

This  amounts to stating a  Large Deviation Principe for measures $\pi ^N$ (with action functional $I(y)= h(y)-\min_{_{\mathcal Y}} h$).  Global minima of $h$  would  then provide the concentration points of $\pi ^N$ in the limit  (the global minimum should be unique, for most values of the parameters).\\ 

 The forecoming analysis will not  directly address this question. 
We will focus on another issue, which is not elucidated in ~\cite{Antunes:08}: Metastability. In  ~\cite{Gibbens:90} metastability is proved to hold  for a rough  one-dimensional diffusion approximation of the model considered.  It tells that as $N$ gets large, the process stays trapped for some long time, of exponential order in $N$, in the neighborhood  of any  stable equilibrium point (regardless of its  asymptotic significance in the invariant distribution).

The present section will show  that both models in  ~\cite{Antunes:08} and ~\cite{Gibbens:90} exhibit a metastable behavior, as a consequence  of the  Large Deviations results of Freidlin and Wentzell (ref ~\cite{Freidlin:84}), here supposed to be still valid under slightly enlarged conditions.

A central notion in  ~\cite{Freidlin:84} is the {\em quasipotential}. For diffusion-like perturbations of dynamical systems, this quantity appears  in  ~\cite{Freidlin:84} as the energy function mentioned  above, under suitable hypothesis, among which uniqueness of the equilibrium point.   It is not clear that this holds for Markov jump processes as those  considered here. Nevertheless, the quasipotential is involved in estimations of exit times that  give evidence  of metastability. 

  Here, it will be moreover shown that the Lyapunov function exhibited in  ~\cite{Antunes:08}  is actually equal to the quasipotential of a slight variant of the model of   ~\cite{Antunes:08} associated with the same  dynamical system.\\

The following very simple observation opens the way to  using  the Large Deviation results of Freidlin and Wentzell. The two models in ~\cite{Antunes:08} and ~\cite{Gibbens:90}, recalled in  ~\ref{subsec:GHK} and ~\ref{subsec:AFRT}, have   a similar structure: For each $N$, they are described by some irreducible Markov processes  $(Y^N(t), t \ge 0)$ on the finite subset $\mathcal Y ^N=\{ y \in\Y: Ny_n \in \N \text{ for all } n\in \X  \}$ of $\Y =\mathcal P(\X)$, with transitions $\ y \longrightarrow  y+\frac{z}{N} $,  where $z$ lives in some finite set $Z \subset \{(z_n) _{n \in \X} \in \Z^{\X}: \sum_n z_n=0\}$ which is independent of $y$ and $N$. Moreover the rate of jump from $y$ to  $ y+\frac{z}{N} $ is given by
\begin{align} \label{eq:rates} Q^N(y,  y+\frac{z}{N})= N \Big( \mu_y(z) +O(1/N) \Big),\end{align} 
  where $(\mu _y)_{y \in\mathcal Y}$  is a family of positive measures on $Z$ with index $y$ in $\Y= \mathcal P (\X)$ such that
 \begin{itemize}
\item $ \mu_y(z)$  is continuous with respect to $y\in \Y$  for fixed $z$,
\item $ \mu_y(z)=0$  whenever $z_n <0$ for some $n\in \X$ such that $y_n=0$, that is, if $y \in \partial \mathcal Y \equiv  \left \{ y\in \mathcal Y: \exists n, y_n=0\right \}$ and if the jump $(y, y+z)$ is directed toward the outside of $\mathcal Y$.
\end{itemize}

In the above expression of jump rates, $O(1/N)$ denotes functions of $N$, $y$ and $z$, whose product by $N$ is uniformly bounded in $y\in \mathcal Y$.

 Jump rates  ~\eqref{eq:rates} tell that  in the neighborhood of some $y$, the process is approximately described by the random walk with jump rates  $\mu _y(z)$, $z \in Z$, rescaled by accelerating time and shrinking space  by the same factor $N$.
\\

For $y \in \mathcal Y$, denote $m_y$ the mean of measure  $\mu _y$ ($\mu _y$ is not a probability measure in general):
\[ m_y = \sum _{z\in Z} z \mu _y(z) \quad (m_y \in  \{(z_n) _{n \in \X}\in \R^{\X}: \sum _{n \in \X} z_n=0\}) .\]
Then for $T<+\infty$, $(Y^N(t))_{0 \le t \le T}$ converges in distribution  as $N$ tends to infinity to the dynamical system:
\begin{align} \label{eq:systdyn} \dot y (t) = m_{y(t)}.
\end{align}

In both examples of interest, $m_y=yL_y$ for some family $(L_y)$ of reversible generators. This will be essential for both interpretations of the Lyapunov function of ~\cite{Antunes:08}, respectively given in Theorem ~\ref{thm:modif} and Proposition ~\ref{prop:lyap}.
 
As  mentioned in Section ~\ref{sec:mean_field}, the models in ~\cite{Gibbens:90}  and ~\cite{Antunes:08} exhibit  three  equilibrium points  (among which two  are stable)  for suitable values of the parameters.\\

If terms $O(1/N)$  in the jump rates  ~\eqref{eq:rates} are omitted  and $\Y$ is replaced by $\R ^d$, the resulting processes belong to a class  introduced by Freidlin and  Wentzell in~\cite{Freidlin:84}, Chapter 5, as \emph{locally infinitely divisible processes}. 
Their context is more general, since the infinitesimal generators  they consider include both jump and diffusion terms, writing  for $C^2$ functions $f$,
\[\Omega^N(f)= \langle m_y,\nabla f(y)\rangle + N \int _{z \in \R^d}  \left [ f(y+ \frac{z}{N}) -f(y) -\frac{1}{N}  \langle z,\nabla f(y)\rangle \right ] d\mu _y(z) +\frac{1}{2N} \sum _{i,j} a_{ij}\frac{\partial ^2 f(y)}{ \partial y_i \partial y_j},\]
where, here and after, $\langle \, , \, \rangle$ denotes the usual scalar product in $\R ^d$. 

The main result in Chapter 5 of ~\cite{Freidlin:84} is that these processes satisfy a Large Deviation Principle on any finite  time interval  $[0,T]$, with scaling coefficient $N$ and   action functional  given by
\[S_{0T}(\varphi)= \left \{ \begin{array}{ll} \int _0 ^T L(\varphi _t, \dot {\varphi} _t) \, dt  &   \text{ for absolutely continuous } \varphi \text{ s.t.\  the integral is well defined},\\ + \infty  & \text{ otherwise}   \end{array} \right. \]
where $L(y,\cdot)$ is the Legendre  transform of some $H(y,\cdot)$  given in our discrete  pure jump case by:
\begin{align} \label{eq:H} H(y,\alpha)= \sum _{z \in Z} \mu _y(z) \Big( e^{\langle \alpha, z \rangle} -1 \Big) \quad (\alpha \in \R^d) .\end{align}
Recall that this  Legendre transform is defined, for $y \in \Y$ and $\beta \in \R ^d$, by
\[ L( y, \beta)= \sup _{\alpha \in \R^d} [\langle \alpha, \beta \rangle -   H(y,\alpha)]. \]

 Large Deviation estimates are then deduced for derived quantities as  the \emph {exit point},  \emph {exit path} and  \emph {exit time} from some domain included in the attraction basin of a stable equilibrium point $y_0$. The estimations for the exit point  and exit time both involve  the \emph {quasipotential} $V$ relative to $y_0$,    which is defined as: 
\[ V(y_0,y)  =\inf \{ S_{0T}(\varphi): T>0,  \varphi \text{ absolutely continuous }, \varphi _0=y_0, \varphi _T=y \}.\]
The exit point from some attracted domain $D$ is shown to be concentrated around the point $y$, if unique,   minimizing  $ V(y_0,y) $ on the boundary $\partial D$ of $D$, while the logarithm of the exit time is close to $N$ times the minimum value of $ V(y_0,y) $ on $\partial D$.   

  A Large Deviation Principle is stated for the invariant distribution,  with  action functional  precisely  given by  $ V(y_0,y) $, but only in the pure diffusion case  and when the equilibrium point $y_0$  is unique.  This cannot thus be applied to the multistable models of interest to us.\\

The main large deviation result for sample paths (hence  all its subsequent  results) relies on  a set of technical hypothesis about the function  $H$ and its Legendre transform $L$. They can be summarized as follows:

I. finiteness of $\sup _ y H(y, \alpha)$;

II. finiteness of $L(y,\beta)$; local in $\beta$, uniform in $y$ boundedness of $L(y,\beta)$ and $\nabla _{\beta}L(y,\beta)$;   strict, uniform in $y$  convexity of functions $L(y,.)$;

III. some specific  equicontinuity property of functions  $L(\cdot,\beta)$.

The last condition is essentially used for replacing arbitrary sample paths by polygons. The mere proof of lower semicontinuity of $S_{0T}$ makes use of  it.
\\

The  models in ~\cite{Antunes:08} and ~\cite{Gibbens:90} differ from the pure-jump processes of ~\cite{Freidlin:84} both through the second order  terms $O(1/N)$ appearing in the jump rates, and through the compact state space $\Y$ standing in place of $\R^d$.

 Dropping one component of $y=(y_n)_{ n \in \X}$  ($\sum _n y_n=1$),  $\Y$ can be identified with a compact subset of $\R^{d}$ with $d= |\X|-1$.  Transitions can then be extended  outside $\Y$ to the entire space $\R^d$. But II and III  then fail to be true, due to the fact that for $y$ on the frontier of $\Y$,  $L(y, \beta)= + \infty$ for  all $\beta$'s outside some cone, depending on $y$ (and not only consisting of those $\beta$'s pointing out of $\Y$).  ~\cite{Wentzell:76} slightly relax conditions I and II, but this is still not sufficient for our purpose. New arguments  need thus  be found  to free from these conditions, in order to deal with sample paths hitting or starting on the boundary of $\Y$. Note that  condition I is satisfied, due to compactness of $\Y$ and continuity of  each $\mu _y(z)$ with respect to $y$.

As for the second order terms  $O(1/N)$, here again the proof of the Large Deviation principle for sample paths (with action functional $S_{0T}$) needs to be adapted. The same change of measure technique might be used, but with $H^N$ and its Legendre transform $L^N$ instead of $H$ and $L$, where $H^N$ is defined analogously to $H$ in  ~\eqref{eq:H},  replacing  $\mu _y(z)$  by  $\mu _y(z) +O(1/N)=1/N Q_N(y,y+z/N)$. The terms $O(1/N)$ should finally  not infer on the  estimations.

All this constitutes a challenging problem which is not addressed in this paper. For the present discussion, it is admitted that all the above mentioned results from Chapter 5 of ~\cite{Freidlin:84}  (that is, the Large Deviation Principle for sample paths, from which all other results derive)  still hold for the two models of interest in this paper.
\\

 In all what follows, the state space $\Y$ is identified with the compact subset of $\R ^d$, where $d=| \X | -1$: $\{ y \in [0, + \infty[^d :  \sum _{i=1}^d y_i \le 1 \}$,  dropping  for example  coordinate $y_0$ of $y = (y_n)_{ n \in \X} \in \Y$. Similarly, the set $Z$ of possible jumps is  identified with a finite subset of $\R^d$. $H(y, \alpha)$ is then  defined    by ~\eqref{eq:H}  for $y \in \Y$ and $\alpha \in \R^d$. This embedding of  $\Y$  in the correct dimension space $\R^d$ in which  the process lives ensures at least that the conditions I to III of ~\cite{Freidlin:84} are satisfied in the interior of $\Y$. 

Let us just mention that the lower semicontinuity of $S_{0T}$ can be proved in our context by describing  $S_{0T}$ as the supremum of functionals  $S^{\varepsilon}_{0T}$ that satisfy the conditions in  ~\cite{Freidlin:84}. Indeed, extend the $\mu_y(z)$ to all $y \in \R^{d}$ so that, for all $z \in Z$, $y \mapsto \mu_y(z)$ is continuous with compact support. Then  introduce the following perturbations of the resulting  $H$: For $\varepsilon >0$,
\[ H^{\varepsilon}(y, \alpha)= H(y,\alpha) + \varepsilon \sum _{z \in Z'}  \Big( e^{\langle \alpha, z \rangle} +  e^{-\langle \alpha, z \rangle} -2 \Big), \quad  (y, \alpha \in   \R^{d}),\]
where $Z'$ is some finite subset of $\R^{d}$   whose generated convex cone is  the whole space. It can then be proved that  the $H ^{\varepsilon}$ and their Legendre transforms  $L ^{\varepsilon}$  satisfy conditions  I to III (following the lines  in the last section of ~\cite{Wentzell:76} for III). According to ~\cite{Freidlin:84}  (theorem 2.1 of Chapter 5) or  ~\cite{Wentzell:76}, the associated $S ^{\varepsilon} _{0T}$ are then lower semicontinuous on the set of continuous paths in $\R ^d$ endowed with the uniform norm topology.

Now  $H ^{\varepsilon}$ clearly decreases to $H$ as $\varepsilon$ decreases to $0$, so that  $L ^{\varepsilon}$ increases and 
\[  L(y, \beta) = \sup _{\alpha} [\langle \alpha, \beta \rangle -   H(y,\alpha)]=   \sup _{\alpha} \sup _{\varepsilon} [\langle \alpha, \beta \rangle -   H ^{\varepsilon} (y,\alpha)] = \sup _{\varepsilon}  L ^{\varepsilon} (y,\beta). \]

It follows by monotone convergence that  $ S _{0T} = \sup _{\varepsilon}S ^{\varepsilon} _{0T}$,  from which the lower semicontinuity of $S_{0T}$ follows.

Notice that   semicontinuity of  $ S _{0T} $ together with condition I  imply (see ~\cite{Wentzell:76}) that the level sets of  $ S _{0T} $  (that is, the sets $\{ \varphi : S_{0T}(\varphi) \le s \}$ for  $s \in [0, + \infty[$) are compact sets in the topology of uniform convergence. This is technically important for proving  the Large Deviation Principle along the classical scheme.
\\

Metastability refers to the large deviation result for exit times from  neighborhoods of  stable equilibrium points.
Theorem 4.3 of Chapter 4 of ~\cite{Freidlin:84} gives the main estimate of such exit times for  Gaussian perturbations of  dynamical systems.    It is then indicated in Chapter 5 how to adapt the proof to a jump-like perturbation of the above described form. In both situations, the result  is derived from the Large Deviation Principle for sample paths.

Recall that an \emph {asymptotically stable} equilibrium point is an equilibrium point $y_0$ such that for any neighborhood $\mathcal N$ of $y_0$, there exists some smaller neighborhood $\mathcal N'$ such that any trajectory initiated in $\mathcal N'$ converges to $y_0$ without leaving $\mathcal N$. 

Also, a domain $D$ is said to be \emph{attracted} to some equilibrium point $y_0$ if any trajectory initiated in $D$  converges to $y_0$ without  leaving $D$.

For the model in ~\cite{Antunes:08}, the Lyapunov function $g$ ensures that all stable equilibrium points (that is, local minima of $g$)  are asymptotically stable.

  From the above discussion, one can deduce from  Chapter 5 of ~\cite{Freidlin:84} the following:

\begin{corollary}
 Assume that the Large Deviation Principle for sample paths, with action functional $S_{0T}$, is valid for the model  in   ~\cite{Antunes:08}. Let $y_0$ be any stable equilibrium point for the limiting dynamics, and let $g$ be the Lyapunov function given by ~\eqref{eq:g}. For any  positive $\delta$, define  $B_{\delta}$ as the connected component of $y_0$ in  the set  $\{y \in \Y: g(y) <g(y_0) + \delta \}$. Define
  \[\tau ^N_{\delta}  = \inf \{t>0: Y^N(t) \notin B_{\delta} \}.\] 
 If  $\delta$ is small enough, then for any $\alpha >0$ and $y \in B_{\delta}$,
\[\lim _{N \to \infty} \P _y(e^{N(V_0- \alpha)} <\tau^N_{\delta} < e^{N(V_0+ \alpha)}) =1\]
where $V_0= \min _{y' \in \partial B_{\delta}} V(y_0,y')$.

\end{corollary}

 \begin{proof}

In order to apply Theorem 4.3 of Chapter 4 of ~\cite{Freidlin:84}, in its modified version discussed in Chapter 5, it must be proved  that $y_0$ is an asymptotically stable equilibrium point for the dynamical system $\dot y =m_y$, that for  small enough  $\delta$,  the  domain  $ B_{\delta}$ is attracted to $y_0$ and has smooth boundary $\partial  B_{\delta}$, and that $\langle n(y),m_y \rangle <0$ for $y \in \partial  B_{\delta}$, where $n(y)$ is the exterior normal at $y$.

As indicated  in Section ~\ref{sec:mean_field}, the critical points of $g$ are isolated, so that the stable equilibria, that is, the  local minima of $g$,  are necessarily {\em strict} local minima. It is then classical (see for example ~\cite{Perko:91}) that, due to the Lyapunov property of  $g$, $y_0$ is asymptotically stable. 

Now choose some open ball $B(y_0,r)$ centered at $y_0$, with radius $r>0$ small enough so  that  $g(y)>g(y_0)$ for all $y \in ( B(y_0,r) \cup \partial B(y_0,r)) \setminus \{y_0 \}$ ($y_0$ is a strict minimum of $g$) and that any trajectory initiated in  $B(y_0,r)$ converges to $y_0$ ($y_0$ is asymptotically stable). Denote $\delta _0= \inf \{g(y)-g(y_0): y \in \partial B(y_0,r) \}$.  Compactness of $\partial B(y_0,r)$ implies that $\delta  _0>0$.

For any $\delta >0$, by the decreasing property of $g$, the set $ \{y \in \Y: g(y) <g(y_0) + \delta \}$ is  invariant under the flow. The same is then true for the connected component $B_{\delta}$, by continuity of trajectories.  In addition, for $\delta \le \delta _0$,  by connectedness  of $B_{\delta}$, $B_{\delta } \subset B(y_0,r)$ (if this were not the case, $B_{\delta}$ would meet $\partial B(y_0,r)$, contradicting  $g(y)-g(y_0) < \delta$ on $B_{\delta}$).
 It  then results that $B_{\delta}$ is attracted to $y_0$. 

It now only remains to prove, still assuming $\delta \le \delta _0$,   that $\partial B_{\delta}$ is smooth and  that  $\langle n(y),m_y \rangle <0$ for $y \in \partial B_{\delta}$, where $n(y)$ the exterior normal of $\partial B_{\delta}$  at $y$. Both properties derive from the fact that  $\partial B_{\delta}$  is  a level surface for $g$. In particular, $ \nabla g(y) = \alpha _y  n(y)$ for some non-negative $\alpha _y$, which is not zero since $y$ is not an equibrium point (take $\delta < \delta _0$ here). Then  $\langle n(y),m_y \rangle= \langle \alpha_y ^{-1} \nabla g(y),m_y \rangle <0$ from the Lyapunov property of $g$. 

The proof is achieved by applying the above  cited theorem of  ~\cite{Freidlin:84} in its version corresponding to the jump-process case of Chapter 5.   
 
 \end{proof}
 
 \begin{remark}
  For the model in ~\cite{Gibbens:90}, the same kind of estimation holds for  exit times from neighborhoods of asymptotically stable  points. Note however that little is known for  this model: Multistability is  obtained  numerically, and  asymptotical stability of  stable points remains to be proved. 
  
  \end{remark}
  
  A companion  corollary can  be stated, telling that the exit point from a neighborhood  $\mathcal N$of $y_0$ is, with high probability as $N \to \infty$, close to the point that minimizes $V(y_0,y)$ on $\partial \mathcal N$, when this one  is unique.  However, this last condition is not guaranteed for $B_{\delta}$. As an example, for  the  modified  model introduced below,  $V(y_0,.)$ and $g$ coincide up to a constant on some neighborhood of $y_0$ (Theorem ~\ref{thm:modif}), so that $B_{\delta}$ is not  a good choice in this case,  $V(y_0,.)$ being \mbox {constant on $\partial B_{\delta}$. }
\\

It will now be proved that $g$ is equal to the quasipotential of a process  obtained from the model in  ~\cite{Antunes:08} by  modifying  the jumps in Section ~\ref{subsec:AFRT} as follows:  Moves  of customers from one node to another are replaced by  departures and arrivals now occuring independently. In other words,  jumps  of the form $ y \longrightarrow  y+\frac{1}{N} (e_{m+f_k}-e_m +e_{n-f_k}-e_n)$ are now split into jumps $ y \longrightarrow  y+\frac{1}{N}  (e_{m+f_k}-e_m)$ and   $ y \longrightarrow   y+\frac{1}{N}  (e_{n-f_k}-e_n)$,  each keeping the original rate. This modifies the Markovian dynamics, but not the limiting dynamical system.

The new transitions and rates are  easily checked: For $1 \leq k\leq K$, $n \in \X$,
\[ y \longrightarrow \left \{ \begin{array}{l}  y+\frac{1}{N} (e_{n+f_k}-e_n) \\ \\
 y+\frac{1}{N} (e_{n-f_k}-e_n)
 \end{array} \right.
\text{at rates}  \begin{array}{l}
  N \Big( (\lambda _k + \gamma _k [I_k,y])y_n {\mathds 1}_{n+f_k \in {\mathcal X}} + O(1/N)  \Big)  \\
\\  N \Big( (\mu _k + \gamma _k) n_k y_n + O(1/N)  \Big)
\end{array} \]

Formally,  the initial Markov process   on $\Y^N$  has jump rates as  in  ~\eqref{eq:rates} with 
\newline $\displaystyle{m_y= \sum _{z \in Z} z \mu _y(z) } =yL_y$ ($y \in \Y$) where  $L_y$  is some infinitesimal generator (namely, that of an $M/M/C/C$ queue with $K$ classes of customers and for  $k=1, \dots ,K$, arrival rates $ \lambda _k+\gamma _k [ I_k,y]$, service rates  $\mu _k +\gamma _k$ and capacity requirements $A_k$). 
So
\[ m_y= \sum _{(m,n) \in \X^2} (y_m L_y(m,n) -y_n L_y(n,m))e_n = \sum _{(m,n) \in \X^2} y_m L_y(m,n) (e_n -e_m). \]
For the modified process just described, transitions are reduced to ``elementary'' jumps of the form  $y \longrightarrow  y+\frac{1}{N} (e_n -e_m)$, with  the associated rates  $N  (y_m L_y(m,n) + O(1/N))$. In other words, for the new process,  \eqref{eq:rates} is still satisfied with $\mu _y(z)$ now defined for  $y \in \Y$ and $z=e_n-e_m$, $n,m \in \X$, by
 \begin{align}  \label{eq:modifrates} \mu _y(e_n-e_m)=  y_m L_y(m,n). \end{align}
 This gives the same limiting vector field $(m_y)_{y\in \Y}$ as for the original process.

\noindent (Note that $L_y(m,n)$ is non zero for $m\neq n$ only if  $n=m \pm f_k$ for some $k=1, \dots ,K$.)

\begin{thm} \label{thm:modif} (i) Let $H$ be the functional defined by ~\eqref{eq:H} for the above modified process. The Lyapunov function $g$ given by ~\eqref{eq:g} satisfies
\begin{align} \label{eq:nabla} H(y, \nabla g(y))=0 \quad (y \in \stackrel{\circ}{\Y)},\end{align}
and 
\[\nabla g(y)=0 \Longleftrightarrow m_y=0.\]

(ii) Let $y_0$ be a stable equilibrium of the dynamical system $\dot y= yL_y$ associated with the model of ~\cite{Antunes:08} or equivalently with its modified version.

Denote   $V(y_0,y)$ the quasipotential  of the modified process,  relative to $y_0$.

The following equality holds on  some neighborhood of $y_0$:
   \[V(y_0,y)=g(y) -g(y_0). \]

\end{thm}

\begin{remark}  \label {remark:nabla}
Before proving the theorem, let us make a remark on equation ~\eqref{eq:nabla}.
Assume that $g$ solves ~\eqref{eq:nabla}  and satisfies
\begin{align} \label{eq:strict} \nabla g(y)=0 \Longrightarrow m_y=0,\end{align}
 then $g$ is a Lyapunov  function for the dynamical system  $\dot y=m_y$.

 Indeed due to  strict convexity of $H(y,\cdot )$,  the  solutions $\alpha$ of $H(y,\alpha)\le 0$ are a  strictly  convex subset $C$ of   $\R^{d}$. Since the boundary $\partial C$  contains $0$, and $m_y = \nabla _{\alpha} H(y, 0) $ is the exterior normal  at  $0$, then  $\langle m_y, \alpha\rangle \le 0$ for any $\alpha $ in $C$, and equality holds only for $\alpha =0$.   In particular, equation  ~\eqref{eq:nabla} implies that $\langle m_y,\nabla g(y)\rangle \le 0$, and  that  $\langle m_y,\nabla g(y)\rangle = 0$ holds only if   $\nabla g(y)=0$, hence only if $m_y=0$ under assumption ~\eqref{eq:strict}.

 \end{remark}

\begin{proof}
(i) $H$ writes:
\[ H(y, \alpha)= \sum _{(m,n) \in \X^2} y_m L_y(m,n) \Big(e^{\langle \alpha, e_n -e_m \rangle} -1\Big). \]
(Note that in order to fit conditions I to III of ~\cite{Freidlin:84} on $\stackrel{\circ}{\Y}$,   $\alpha$  should vary  in $\R^{|\X|-1}$, $e_n-e_m$ being replaced by its projection on $\R^{|\X|-1}$, and scalar products being understood in  $\R^{|\X|-1}$. But for computing $ H(y, \nabla g(y))$,  it is easily checked that the result in unchanged if one keeps the original definition of $g$ as a function of $y \in \R^{|\X|}$ and uses  the scalar product in $\R^{|\X|}$).  

Denote $\nu _{\rho (y)}$, as in Section ~\ref{sec:mean_field},  the reversible distribution associated with generator $L_y$, and set $q_y(m,n)= \nu _{\rho (y)}(m) L_y(m,n)$ so that $q_y$ is symmetric in $(m,n)\in \X ^2$. Then
  \begin{multline*} H(y, \nabla g(y))= \sum _{(m,n) \in \X^2} y_m L_y(m,n) \, e^{-\langle \nabla g(y),  e_m \rangle} \Big(e^{\langle \nabla g(y), e_n  \rangle} - e^{\langle \nabla g(y), e_m  \rangle}\Big) \\
= \frac{1}{2} \sum _{(m,n) \in \X^2} \Big(y_m L_y(m,n) \, e^{-\langle \nabla g(y),  e_m \rangle} - y_n L_y(n,m) \, e^{-\langle \nabla g(y),  e_n \rangle} \Big) \Big(e^{\langle \nabla g(y), e_n  \rangle} - e^{\langle \nabla g(y), e_m  \rangle}\Big) \\
= \frac{1}{2} \sum _{(m,n) \in \X^2} q_y(m,n) \Big(\frac{y_m}{\nu _{\rho (y)}(m)}  \, e^{-\langle \nabla g(y),  e_m \rangle} - \frac{ y_n}{\nu _{\rho (y)}(n)}  \, e^{-\langle \nabla g(y),  e_n \rangle} \Big) \Big(e^{\langle \nabla g(y), e_n  \rangle} - e^{\langle \nabla g(y), e_m  \rangle}\Big) 
. \end{multline*}
The result follows from the fact that  ${\langle \nabla g(y), e_n  \rangle}=  \log \frac{ y_n}{\nu _{\rho (y)}(n)} +1- \log Z( \rho (y))$, as can easily be computed from  equation ~\eqref{eq:g}, so that
$ \frac{ y_n}{\nu _{\rho (y)}(n)}  \, e^{-\langle \nabla g(y),  e_n \rangle}$ is independent of $n\in \X$.

(ii) is then essentially  a consequence of a result by Freidlin and Wentzell, stated in ~\cite{Freidlin:84} as Theorem 4.3 of Chapter 5. Since  a ``local'' variant of this result is actually  needed, an independent proof is given, for completeness.
 Recall that 
  \[V(y_0,y)= \inf \left \{ S_{0T}(\varphi): T>0,   \varphi \text{ absolutely continuous}, \varphi _0=y_0, \varphi _T=y \right \} ,\]
  and consider any $T>0$ and any absolutely continuous $\varphi$  on $[0,T]$ such that $ \varphi _0=y_0, \varphi _T=y$. Reversing  time variable $t$ in the integral 
  $S_{0T}(\varphi)= \int _0 ^T L(\varphi _t, \dot {\varphi} _t) \, dt $ gives   $S_{0T}(\varphi)= \int _0 ^T L(\psi _t, -\dot {\psi} _t) \, dt $ where $\psi _t= \varphi _{T-t}$ for $0 \le t \le 0$.
  
  Introduce the following family of measures $(\tilde \mu _y)_{y \in \Y}$  on $-Z$, defined by
  \[ \tilde \mu _y (z) = e^{-  \langle \nabla g(y),z \rangle} \mu _y(-z) \quad (z \in- Z). \]
  (Considering $\mu _y$ as the jump length distribution  for  the  local random walk  approximating  $Y^N$ in the neighborhood of $y$, $\tilde \mu _y$ corresponds to the reversed random walk with respect to the measure  $ ( e^{-  \langle \nabla g(y),z \rangle})_{z \in Z}$, which is stationary  by ~\eqref{eq:nabla} and  ~\eqref{eq:H}.)
  
  Define $\tilde H$ and $\tilde L$ associated with $(\tilde \mu _y)$ in the same way as $H$ and $L$ with $(\mu _y )$. The following relations are easily checked: for any  $y \in \Y$ and $\alpha \in \R ^{\X}$,
  \[ \tilde H(y, \alpha)= H(y, \nabla g(y) - \alpha) \quad \text{ and }  \quad  \tilde L(y, \beta)= L(y, -\beta) + \langle \nabla g(y), \beta \rangle. \]
  One gets
  \[S_{0T}(\varphi)= \int _0 ^T \tilde L(\psi _t, \dot {\psi} _t) \, dt -\int   _0 ^T \langle \nabla g(\psi_t), \dot {\psi} _t \rangle \, dt =\int _0 ^T \tilde L(\psi _t, \dot {\psi} _t) \, dt + g(y) -g(y_0) \]
    since $\psi _0= y$ and $\psi _T=y_0$. 
    
 All that is left to prove now is that  for $y$ close enough to $y_0$,
   \begin{align}  \label{eq:inf} \inf  \left \{ \int _0 ^T \tilde L(\psi _t, \dot {\psi} _t) \, dt: T>0,   \psi  \text{ abs.~continuous}, \psi _0=y, \psi _T=y_0  \right \} =0 .
    \end{align}
   Using  Remark~\ref{remark:nabla}, it can be shown that   $g$ is a Lyapunov function for the ``locally reversed" dynamical system
   \[\dot y= \tilde m_y \quad \text{ where } \tilde m_y= \sum _{z\in -Z} z \tilde \mu _y(z). \]
   Indeed $\tilde H(y ,\nabla g(y) )= H(y, 0)=0$ for all $y$
   and ~\eqref{eq:strict} is   satisfied with $\tilde m_y$ in place of $m_y$ since $\nabla g(y)=0 $  both implies that $m_y=0$   by (i) and  that $\tilde m_y=- m_y$ from the definition of $\tilde \mu _y$.
   
   One can deduce from this that $y_0$ is also a stable equilibrium point for the dynamical system $\dot y= \tilde m_y$. As a consequence, there exists a neighborhood $\mathcal N$ of $y_0$ such that  all trajectories initiated in $\mathcal N$ converge to $y_0$ at infinity.  Assume from now on that $y \in \mathcal N$ and consider the trajectory $\psi$ initiated at $y $, then $\lim _{t \to + \infty} \psi _t = y_0$ and  $ \int _0 ^{+ \infty} \tilde L(\psi _t, \dot {\psi} _t) \, dt  =0$.
   
   It is then easy to derive  that the infimum in ~\eqref{eq:inf} is non-positive, using  boundedness of $\tilde L(y, \beta)$ on compact subsets of $\stackrel {\circ}{\Y} \times \R ^{\X}$. The result follows since this infimum is clearly non-negative (as  $\tilde L \ge 0$).
   
\end{proof}

\begin{remark} \label{remark:balance}
We conclude this section with two heuristic remarks.

  First, ~\eqref{eq:nabla} can be seen as  the limiting balance equation satisfied by  $\pi ^N$ if the approximation
\begin{align} \label{eq:pi} \pi ^N(y) \approx \frac {e^{-N g(y)}}{Z_N} \quad \text{ as } N \to \infty \end{align}
is valid with enough accuracy for some differentiable $g$.
Indeed, \mbox{the balance equations}
\[   \forall y \in \mathcal Y ^N \quad \sum _{z \in Z} \Big( \pi ^N(y-\frac{z}{N}) Q^N(y-\frac{z}{N},y) - \pi ^N(y) Q^N(y,y+\frac{z}{N}) \Big)=0,\] 
  become 
$ \quad \displaystyle{  \frac {e^{-N g(y)}}{Z_N}  \sum _{z \in Z} \mu _y(z) \Big( e^{\langle \nabla g(y), z \rangle} -1 \Big)=0} \ $ under ~\eqref{eq:pi}, 
since
\[  \pi ^N(y-\frac{z}{N})/\pi ^N(y) \approx  e^{\langle \nabla g(y), z \rangle} \quad \text{ and } \quad  Q^N(y-\frac{z}{N},y)\approx  Q^N(y,y+\frac{z}{N}) \approx \mu _y(z).\]
This intuitively confirms the relation between solutions of ~\eqref{eq:nabla} and the quasipotential stated in (ii) of  Theorem ~\ref{thm:modif} (or in Theorem 4.3 of Chapter 5 of  ~\cite{Freidlin:84}), since in the unique equilibrium case, the quasipotential is  expected to be the action functional associated to the invariant distribution.

Secondly, the above proof of (i) suggests that for the modified process,  the measure $(e^{-Ng(y)})_{y\in \Y^N}$  solves the  {\em local}  balance equations in the limit. Indeed,  from ~\eqref{eq:modifrates} and  the proof of (i), $g$  solves 
\begin{align} \label{eq:rev} \mu _y(z) -   \mu _y(-z)  e^{-\langle \nabla g(y), z \rangle}=0 \quad (y \in \Y, z \in Z).\end{align}

This relation appears as the limiting identity obtained from
\[  e^{-Ng(y)} Q^N(y,y+\frac{z}{N}) = e^{-Ng(y+\frac{z}{N})} Q^N(y+\frac{z}{N},y)\]
via the same approximations as in the first part of the present remark. 
  The modified process can thus be considered as asymptotically reversible. This is a particularity of the dynamical system in ~\cite{Antunes:08}. For the analogous  modified version of the process in ~\cite{Gibbens:90}, no $g$ satisfies equation ~\eqref{eq:rev}.

\end{remark}

\section{Lyapunov function and  relative entropy} \label{sec:lyapunov}
 This section is devoted to a rereading of the Lyapunov function $g$, given by ~\eqref{eq:g},  for  the model in  ~\cite{Antunes:08}. It is  connected to a well known decreasing property of  relative entropy.  Two other models can then be introduced, for which a similar relative entropy argument  for constructing a Lyapunov function applies. Convergence of their stationary measure to a Dirac mass can then be derived.

\subsection{A reformulation of the Lyapunov function} \label{subsec:g} 

Section ~\ref{sec:quasipot} has provided an interpretation  of $g$ in terms of the quasipotential associated with  another Markov  random  perturbation of the same dynamical system.
The  Lyapunov property of  $g$  is now given an interpretation  in terms of relative entropy.

 Underlying both interpretations, it appears that the model in  ~\cite{Antunes:08} is very particular, due to  the quantities $[I_k,y]$ that drive the $M/M/C/C$  generators $L_y$ involved in equation ~\eqref{eq:S}.\\

For $y,y' \in \Y$, two probability distributions on $\X$,  the relative entropy of  $y$ with respect to $y'$ is defined as:
\[h(y|y') =\sum _{n \in \X} y_n \log \frac{y_n}{y'_n}. \]
It is non-negative,  and finite if $y' \in \,  \stackrel {\circ}{\Y} = \{ y \in \Y: y_n >0 \text{ for all } n \in \X \}$. 

Relative entropy appears in the following, easily  checked, expression of  $g$:
\begin{align} \label{eq:entg} g(y) = h(y|\nu _{\rho (y)})-\log Z(\rho(y)) +\sum _{k=1}^K \psi _k( [ I_k,y ]) \quad  \quad (y \in \Y) 
\end{align}
where $\nu _{\rho (y)}$ is the Erlang reversible distribution of $L_y$ (here $\displaystyle{\rho _k(y)= \frac{ \lambda _k+\gamma _k [ I_k,y] }{\mu _k + \gamma _k}}$ for $1 \le k \le K$) and:
 \[\psi _k(x)= \int _0^x \frac{\gamma _k u}{\lambda _k+ \gamma _k u} \, du \quad \quad (k=1, \dots ,K).\]
Note that, though no mention of relative entropy is made in ~\cite{Antunes:08}, it is there noticed that critical points of $g$ on $\Y$ correspond, through $\rho \mapsto \nu _{\rho}$, to critical points of some function of $\rho$ (given by the two last terms in the right hand side of ~\eqref{eq:entg}). This was the  clue for a dimension reduction argument.
  
The Lyapunov property of $g$ will  be explained from the classical decreasing  property of the relative entropy between the distribution at time $t$ of some Markov ergodic process and its invariant distribution.  Entropy dissipation and its quantification  are widely present in the literature. Classicaly,  estimating the entropy dissipation is an alternative to logarithmic Sobolev inequalities for obtaining  exponential rates of decay to equilibrium (see for example ~\cite{Caputo:07}).  In a different context,   ~\cite{Yau:91} introduces a method for deriving hydrodynamical limits (see also   ~\cite{Olla:93}, ~\cite{Kipnis:99}) by  controlling the variations of the relative entropy between the current distribution  and the so-called local equilibrium,  varying in time.   Our  situation differs from the usual ones, in the sense that  the generator itself varies in time: $y$ solves  $\dot y (t)=y(t)L_{y(t)}$. However, this can be compensated by adding appropriate terms. This is made possible  because   $\rho _k(y)$'s depend on $y$  only  through  quantities   $[ I_k,y ]$, that naturally appear in the expression of  $ h(y|\nu _{\rho (y)})$.  This method for building a Lyapunov function is thus restricted to very special dynamics. A  different example will be given in ~\ref{subsec:closed}, for which the method applies due to invariance of $[I,y]$ along the  flow.

In  the following proof, the reversibility of  generators $L_y$ manifests itself,  as it is classical, through  the Dirichlet form.\\

\begin{prop}  \label{prop:lyap} Assume that for $y \in \Y$, $L_y$ is the infinitesimal generator of an $M/M/C/C$ queue with $K$ classes of customers  having capacity requirements $A_k$ and arrival-to-service rate ratios    $ \rho _k(y)= \varphi _k ( [ I_k,y])$ for $k=1, \dots ,K$, where the $\varphi _k$  are positive  $C^1$ functions on $\R$. Set $\rho(y)=(\rho _k(y), 1\le k \le K)$ for $y \in \Y$.

Then the following function is a Lyapunov function for  dynamical system ~\eqref{eq:S}:
\[g(y)= h(y| \nu _{\rho(y)}) -\log Z(\rho(y)) +\sum _{k=1}^K \psi _k( [ I_k,y ]), \]
where, for $k=1, \dots ,K$,  $\psi _k$ is some primitive of $\displaystyle{x \mapsto x \frac{\varphi ' _k(x)}{ \varphi _k (x)}}$.

\end{prop}

\begin{proof}
Using the definition of relative entropy together with ~\eqref{eq:erlang} gives for $y \in \Y$,
\[g(y)= \sum _{n \in \X} y_n \log \frac{n! y_n}{ \rho (y) ^n } + \sum _k \psi _k ( [ I_k,y ])=  \sum _{n \in \X} y_n \log (n! y_n)  + \sum _k \Big( \psi _k ( [ I_k,y ]) - [ I_k,y ] \log  \rho _k (y) \Big).\]
  Due to the  relation,  for $k=1, \dots ,K$, between  $\psi _k$ and the  $\varphi _k$ satisfying $\rho  _k(y)=\varphi _k ( [ I_k,y ])$, derivation with respect to $y_n$ $(n\in \X)$ simply  yields:
\begin{align} \label{eq:nablag} \frac{\partial g(y)}{\partial y_n}= \log\frac{n! y_n}{\rho (y)^n } +1  .\end{align}
It must be proved that $  yL_y \nabla g(y)  \le 0$ for all $y \in \Y$, equality holding only when $yL_y=0$.
Now $  yL_y \nabla g(y) $ can be expressed using the following identity due to  reversibility of generator $L_y$ with respect to distribution $\nu _{\rho (y)}$:
 for $u \in \R ^{\X}$,
 \begin{align} \label{eq:u} yL_y u  =  - \frac{1}{2} \sum _{(m,n) \in \X ^2} q_y(m,n) \Big(\frac{y_m}{\nu _{\rho(y)}(m)}   -\frac{y_n}{\nu _{\rho(y)}(n)} \Big)( u_n -u_m)
\end{align}
where $q_y(m,n)= \nu _{\rho(y)}(m) L_y(m,n)$ is non-negative and symmetric in $(m,n)$.\\
(The last term is the Dirichlet form associated to $L_y$  evaluated at vectors $\displaystyle{ \Big (\frac{y_n}{ \nu _{\rho (y)}(n)} \Big)}$ and $u$.)
Then using ~\eqref{eq:nablag},
\[  yL_y \nabla g(y)  =  - \frac{1}{2} \sum _{(m,n) \in \X ^2} q_y(m,n) \Big(\frac{y_m}{\nu _{\rho(y)}(m)}   -\frac{y_n}{\nu _{\rho(y)}(n)} \Big)\Big(\log  \frac{y_m}{\nu _{\rho(y)}(m)}   - \log \frac{y_n}{\nu _{\rho(y)}(n)}\Big). 
\]
This shows that $  yL_y \nabla g(y)  \le 0$ for all $y$. It can be zero only if  $y_m/\nu _{\rho(y)}(m)   =y_n/\nu _{\rho(y)}(n)$ for all pair $m,n$ such that $q_y(m,n)>0$. By irreducibility of $L_y$, this is possible only if $y_n/\nu _{\rho(y)}(n)$ does not depend on $n$, which means that $y=\nu _{\rho(y)}$.

\end{proof}

\subsection{Some Statistical Mechanics formalism} \label{mecastat}
Measures $\nu _{\rho}$ with $\rho \in ]0,+\infty[^K$, that include the fixed points of all our models, can be written in the  following Gibbs form:
\[ \nu _{\rho}(n)= \frac{1}{Z(\rho)}\exp \sum _k (n_k \log \rho _k - \log n_k!)=  \frac{1}{Z(\theta)}\exp (\langle \theta, n \rangle - \log n!)   \quad (n\in \X).\]
where $\theta = (\theta _1, \dots ,\theta_K)$ is defined by $\theta _k = \log \rho _k$  ($1 \le k \le K$) and  appears as the natural Gibbs parameter. It  is then convenient to re-parametrize the family $(\nu _{\rho})$ as $(\nu _{\theta})$ for $\theta \in \R ^K$ (abusively  writing $Z(\theta)$ for $Z(\rho)$). This  labelling is  \mbox{clearly one-to-one.}

We now address  the problem of minimizing the relative entropy distance of a given probability measure $y$ on $\X$ to the set of $\nu_{\theta}$'s.
 It is related  to some classical results in Statistical Mechanics. Yet, our arguments may  not  be standard. A more complete overview can be found in ~\cite {Ellis:85}, where this  problem underlies some contraction principles related to  Large Deviations of i.i.d.~random vectors.
Only the case $K=1$ is relevant for the present paper, since the two next models are one class, but  Proposition ~\ref{prop:prox}, is worth mentioning in the multidimensional case  for its own sake, or for possible extension of  section ~\ref{subsec:open}  to several classes.  

First, it is   classical  that derivating the ``free energy" $\log Z$ gives the ``magnetization", that is, the expectation of the Gibbs measure: 
\[ \nabla  \log Z(\theta)=[ I, \nu _{\theta} ]. \]
Note that for $Z$ regarded as a function of $\rho$, it gives
\[ \frac{\partial \log Z(\rho)}{\partial \rho _k}= \frac{1}{\rho _k}[ I_k ,\nu _{\rho} ] = 1- B_k(\rho) \quad (k=1, \dots ,K), \]
using the well known relation $[ I_k ,\nu _{\rho} ] = \rho _k (1- B_k(\rho))$. Here for $k=1, \dots ,K$, 
\[B_k(\rho)= \sum _{n \in \X: n+f_k \notin \X} \nu _{\rho}(n)\]
 is the so-called ``blocking probability" corresponding to class $k$ for parameter $\rho$, that is, the probability  that a new class $k$ customer is rejected in an $M/M/C/C$ queue with load $\rho$ in its stationary regime.\\

Next, $\log Z(\theta)$ is strictly convex with respect to $\theta$: This can be shown using the relative entropy  between two Gibbs measures $\nu _{\theta}$. Indeed for $\theta, \theta ' \in \R ^K$
\begin{align} \label{eq:gibbs_entropy}  h(\nu _{\theta} | \nu _{\theta '}) = \sum _{n \in \X} \nu _{\theta}(n) \log \frac{\nu _{\theta}(n)}{\nu _{\theta '}(n)} = \log \frac{ Z(\theta ')}{Z(\theta)} + \langle [I, \nu _{\theta}], \theta '-\theta \rangle.
\end{align}
which rewrites
\[ \log Z(\theta ') -\log Z(\theta) = \langle \nabla \log Z_{| \theta}, \theta '-\theta \rangle +  h(\nu _{\theta } | \nu _{\theta '}) .\]
This shows that the non-negative quantity $h(\nu _{\theta } | \nu _{\theta '}) $  (positive if $\theta ' \neq \theta$)  measures the difference between the graph of $\log Z$ and its tangent hyperplane at $ \theta$. It proves strict convexity of $\log Z( \theta)$.

Properties of the relative entropy of some $y \in \Y$ with respect to some $\nu _{\theta}$ can then be derived. For $y \in \Y$ and $\theta \in \R ^K$:
\begin{align} \label{eq:entropy} h(y| \nu _{\theta }) = \sum _{n \in \X} y_n \log \frac{y_n}{\nu _{\theta }(n)} = \log Z(\theta)  +\sum _n y_n \log (n!y_n) -\langle [I, y], \theta  \rangle,
\end{align}
so that for fixed $y$, the relative entropy $h(y| \nu _{\theta }) $ is  also strictly convex in $\theta$ (as the sum of $\log Z(\theta)$ and an affine function of $\theta$). 

It is not difficult to show that if $y \in \stackrel{\circ}{\Y}$ (that is $y \in\Y$ and $y_n>0$ for all $n$), then $h(y| \nu _{\theta }) $ tends to infinity as $\| \theta \| \to + \infty$: Indeed,  for any $\theta \in \R ^K$, the following inequalities hold:
  \begin{multline*} \log Z(\theta) -\langle [I, y], \theta  \rangle  \ge   \max _{n \in \X} \, (\langle \theta, n \rangle - \log n!)  -\langle [I, y], \theta  \rangle \\
   \ge \max _{n \in \X} \, \langle n, \theta \rangle -  \langle [I, y], \theta  \rangle - \max _{n \in \X} (\log n!). 
   \end{multline*}
Now setting $h_y(\theta)=  \max _{n \in \X}  \langle n, \theta \rangle -  \langle [I, y], \theta  \rangle = \sum _m y_m ( \max _n   \langle n, \theta \rangle - \langle m, \theta \rangle)$,  $h_y(\theta)$ is positive for any $ \theta \ne 0$ since all $y_n$ are positive and the terms  $ \langle m, \theta \rangle, m \in \X,$ cannot be all equal (recall that, since  $A_k \le C$  for $k =1, \dots , K$, the set $\X$ contains $0$ together with  the canonical vectors $f_1, \dots , f_K$).  Since $h_y$ is continuous on $\R^K$, it is bounded from below by some positive constant on the unit sphere of $\R^K$. Then using  the fact that $h_y(\theta) = \| \theta \| \, h_y( \theta / \| \theta \|)$ for all non-zero $\theta$, it results that $h_y$, and hence   $h(y| \nu _{\theta }) $, tends to infinity as $\| \theta \| \to + \infty$.

 As a consequence, $h(y| \nu _{\theta }) $ attains one unique minimum on $\R ^K$ at some value denoted $\bar {\theta} (y)$. Derivating $h(y| \nu _{\theta }) $ with respect to $\theta$ gives  that $\bar {\theta} (y)$ is the unique solution of
\[ [I, \nu _{\theta }] = [I, y]. \]
(Note that unicity also appears on the following relation, itself derived from ~\eqref{eq:gibbs_entropy}:
\[ h(\nu _{\theta} | \nu _{\theta '})  + h(\nu _{\theta '} | \nu _{\theta }) = \langle [I, \nu _{\theta '}]- [I, \nu _{\theta}], \theta '-\theta \rangle,\]
which proves  that equality  $[I, \nu _{\theta '}]= [I, \nu _{\theta}]$ is possible only if $\nu _{\theta '}=\nu _{\theta }$, i.e., $\theta  '= \theta$.)

One  gets  the following result: 
\begin{prop} \label{prop:prox}
For any $y \in \stackrel{\circ}{\Y}$, there exists one unique  $\theta \in \R ^K$, denoted $\bar {\theta} (y)$ such that $[I, \nu _{\bar {\theta} (y)}] = [I, y]$.
This $\bar {\theta }(y)$ minimizes  $\theta \mapsto h(y| \nu _{\theta }) $, and moreover  satisfies
\[h(y| \nu _{\theta  })  = h(y| \nu _{\bar {\theta} (y)})  + h(\nu _{\bar {\theta} (y)}| \nu _{\theta }) \quad \text{ for any } \theta \in \R ^K. \] 

\end{prop}

\begin{proof}
Only  the last relation is left to prove, but it is a direct consequence of ~\eqref{eq:entropy}, ~\eqref{eq:gibbs_entropy} and $[I, \nu _{\bar {\theta} (y)}] = [I, y]$.
\end{proof}

\begin{remark}
From the additive formula  in the previous proposition, it also results that  for any fixed  $ \theta, {\bar {\theta} } \in \R^K$,  the measure $ \nu _{\bar {\theta} }$  minimizes   $y \mapsto h(y| \nu _{\theta }) $ on the set of probability measures $y$   having mean $[ I, \nu _{\bar {\theta} } ]$.  Theorem VIII.4.1 in ~\cite {Ellis:85} shows that this  minimum value  can  be read as some Legendre-Fenchel transform. 
\end{remark} 

As a corollary, getting back to the $\rho$-parametrization, one gets that the mapping $\rho \mapsto [I, \nu _{\rho}]$ is one to one from the set $]0 ,+ \infty[^K$ onto the set $\{ \sum _{n \in \X}  ny_n, y \in \stackrel{\circ}{\Y } \}$. When $K=1$, the  last set is simply the interval $]0,C[$, but for larger $K$ it is not easy to characterize. Like the convex hull of $\X$: $\{ \sum _{n \in \X}  ny_n, y \in \mathcal \Y \}$,  it  arithmetically depends on integers $C$ and $A_k$'s  in some intricate manner. In particular it does not  coincide in general with the set $\{ m\in ]0, +\infty[^K: \sum _k A_k m_k <C \}$.

\begin{remark}
The term  $\langle [I, y], \theta  \rangle$ in relation ~\eqref{eq:entropy} explains the particularity of  the case  where $\rho  _k(y)$'s are functions of  $[I _k, y]$'s,  as considered in Proposition ~\ref{prop:lyap}. Another  nice situation, which is the case for the next model,  is when  $[I, y]$ remains constant along the trajectories of the dynamical system.

\end{remark}

\subsection{A closed system}\label{subsec:closed}
The next model again describes a system of $N$ nodes with the same capacity $C$, but here no rejection can occur: Customers are directly routed towards non saturated nodes (one can imagine that they  are randomly rerouted as many times as necessary, at a null time cost, so as to find some node having the required free capacity).

In this model, there are no external arrivals nor departures.  $M$ customers are present for ever in the system, with $M<NC$. Each customer spends an exponential time  with mean one at each visited node, after which he chooses uniformly one new node among those, different from the current node, that are not saturated (if the current position is the only non saturated one, the customer does not move). All exponential variables  and choices of successive nodes are assumed independent. The model is analyzed through the following asymptotics: $N$ and $M=M(N)$ tend to infinity with $M(N)/N$ converging to some $\lambda \in ]0,C[$.

For fixed $N$ and $M$  with $M<NC$, the empirical measure process is defined as in ~\eqref{eq:Y} and here denoted $(Y_M^N(t), t\ge 0)$. It is a  Markov jump process on the finite subset $\Y_M^N = \{y=(y_n, 0 \le n \le C) \in \Y: Ny_n \in \N \text{ for } n= 0, \dots ,C \text{ and } [I,y ]=\sum _{n=1}^C ny_n=M/N \}$ of $\Y=   \{y=(y_n, 0 \le n \le C) \in [0, + \infty[^{C+1}: \sum _{n=0}^C y_n=1\} $. The transitions are the following ($e_n$ still denotes the  $n^{th}$ unit vector in $\R^{C+1}$):
\begin{align} \label{eq:closed} y \longrightarrow  y+\frac{1}{N} \big( e_{m+1}-e_m  +e_{n-1}-e_n\big) \text{  at rate }  N ny_n \frac{Ny_m-\indicator {m=n}}{N(1-y_c)},
\end{align}
for $0 \le n \le C$ and  $0 \le m \neq n \le C-1$.
(This corresponds to some move from some node in state $n$ to some node in state $m$. Note that due to the condition $M<NC$, $y_C<1$ for $y \in \Y^N_M$.)
It is clear that $Y_M^N$ is irreducible and thus admits one unique invariant distribution.

\begin{remark}
Extension of this model to the case with $K \ge 2$ classes of customers (with capacity requirements $A_1, \dots ,A_K$) is problematic, since in this case, the Markov process analogous to $Y^N_M$ is possibly non irreducible. For example, consider the case of $N$ nodes, each of capacity $6$, and $2N$ customers,  $N$  of each of two classes with $A_1=3$ and $A_2=2$. The distribution with  one customer of each class at each node does not communicate with any other configuration (e.g., for  $N$  even, with the configuration with $N/2$ nodes occupied by two class  $1$ customers and the  $N/2$ others by two class $2$ customers), because only one unit of free capacity is available at each node. And yet, the total free capacity goes to infinity  proportionally to $N$.

This problem could be solved   in some sense by replacing this closed system by the open  (irreducible) one next introduced  in ~\ref{subsec:open}, extended to the multiclass case (see Remark  ~\ref{remark:multiclass} below), that should have the same limiting dynamics in restriction to $\Y _{\lambda}$.
\end{remark}

Analogously to the previous examples, it can be  shown that as $N,M$ go to infinity with $M/N $ converging to $\lambda$, the  process $(Y^N_M(t), 0 \le t \le T)$ converges  in distribution, for any finite  $T$,  to the solution  with initial value $y(0)$ of the following  differential system of equations: 
\begin{multline} \label{eq:closed_dyn} y_n'(t)= \frac{\lambda}{1-y_C} y_{n-1}(t) \mathds 1_ {n\ge 1} + (n+1) y_{n+1} (t) \mathds 1_ {n\le C-1}
 \\ - \Big(\frac{\lambda}{1-y_C} \mathds 1_ {n\le C-1} +n \Big) y_n(t) \quad (n=0, \dots ,C).
\end{multline}
 provided that $Y^N_M(0)$ converges in distribution to $y(0)$.

Since $M/N$ converges to $\lambda$, the initial point $y(0)$ necessarily belongs  to the set $\Y_{\lambda}= \{y \in \Y: \sum _n ny_n=\lambda \}$. As the next proposition will show, unsurprisingly, 
 the set $\Y _{\lambda}$ is invariant under the above system. Now since $\lambda <C$,  $y_c<1$ for  any $y\in \Y_{\lambda}$, so that for $y(0) \in \Y_{\lambda}$, the solution of \eqref{eq:closed_dyn} is  defined on the whole time axis.

This system rewrites as ~\eqref{eq:S} where now, for $y \in \Y _{\lambda}$,  $L_y$ is the rate matrix of an $M/M/C/C$ queue with arrival rate $\rho(y)=\frac{\lambda}{1-y_C} $ and service rate $1$.

\vspace{1mm}

Equilibrium points are again characterized by ~\eqref{eq:fix}, where $\rho(y)= \frac{\lambda}{1-y_C}$ and $\nu _{\rho}$ is defined by ~\eqref{eq:erlang} for $\rho>0$. This fixed point equation has one unique solution. Indeed, since $0< \lambda < C$, Proposition ~\ref{prop:prox} above shows that there is a  unique $\rho \in ]0,+ \infty[$ such that $[ I,\nu _{\rho}]= \lambda$.
Denote  it $\rho _{\lambda}$. Equation ~\eqref{eq:fix} rewrites: $y=\nu _{\rho}$ with $\rho= \frac{\lambda}{1-B(\rho)} $ and $B(\rho)= \nu _{\rho}(C)$, or equivalently, using the relation $[ I,\nu _{\rho} ]=\rho (1-B(\rho)) $: $[ I,\nu _{\rho } ]= \lambda$, so that $\nu _{\rho _{\lambda}}$ is the unique equilibrium point for ~\eqref{eq:S} on $\Y _{\lambda}$.

\begin{prop} \label{prop:closed} Fix $\lambda \in ]0,C[$ and assume that for $y \in \Y _{\lambda}$,  $L_y$ is the rate matrix of an $M/M/C/C$ queue with arrival rate $\rho(y)=\frac{\lambda}{1-y_C} $ and service rate $1$.  Then

(i)  The set $\Y_{\lambda}$ is invariant under the dynamical system ~\eqref{eq:S}.

(ii)   $g(y)=h(y| \nu _{\rho _{\lambda}}) $ is a Lyapunov function for this dynamical system on $\Y_{\lambda}$.

(iii) The sequence of invariant probability distributions $\pi ^N$ of processes $Y^N_{M(N)}$ converges weakly to the Dirac mass at $\nu _{\rho _{\lambda}}$ as $N,M(N) \to \infty$ with $M(N)/N \to \lambda$. 

\end{prop}

\begin{proof}
(i) Set $f(y)= \sum _n ny_n$ ($y \in \Y$).  Then $\nabla f(y)=(n, 0 \le n \le C)$.  The derivative of $f$ along the dynamical system ~\eqref{eq:closed_dyn} is given by $yL_y \nabla f(y)$. Using ~\eqref{eq:u}  with the present reversible generators $L_y$ and $u_n=n$ for  $0 \le n \le C$ gives
\[ yL_y \nabla f(y)= - \sum _{n=0}^{C-1} q_y(n,n+1) \Big(\frac{y_{n+1}}{\nu _{\rho(y)}(n+1)}   -\frac{y_n}{\nu _{\rho(y)}(n)} \Big),\]
where $q_y(n,n+1)=  \nu _{\rho(y)}(n)L_y(n,n+1)=  \nu _{\rho(y)}(n+1)L_y(n+1,n)$ so that
\begin{multline}  \label{eq:mean}  yL_y \nabla f(y)= -\sum _{n=0}^{C-1} \Big( y_{n+1} L_y(n+1,n) -y_nL_y(n,n+1) \Big)\\
=- \sum _{n=0}^{C-1} \Big( (n+1) y_{n+1}  - \rho(y) y_n \Big)= (1-y_C) \rho(y) - \sum _n ny_n =0 \end{multline}
for $y \in \Y _{\lambda}$, which proves invariance of $\Y _{\lambda}$ under ~\eqref{eq:closed_dyn}. 

(ii)  Using relation ~\eqref{eq:entropy} together with invariance of $[I,y ]$ under the present dynamical system gives for $y \in \Y _{\lambda}$
\[  yL_y \nabla g(y)= yL_y \Big(\log (n!y_n)+1\Big) _{0 \le n \le C}= yL_y \Big(\log \frac{n!y_n}{Z(\rho (y))}\Big)_{0 \le n \le C} , \]
(as shown by ~\eqref{eq:u},  $ yL_y u $ is inchanged through adding some constant to   $u$).
Then, again using invariance of $[ I,y ]$ under the dynamics, which writes $ yL_y (n,0 \le n \le C)  =0$, one also has
\[  yL_y \nabla g(y) =  yL_y \Big(\log \frac{n!y_n}{ \rho (y) ^n Z(\rho (y))}\Big) _{0 \le n \le C} = yL_y \Big(\log \frac{y_n}{ \nu_{\rho (y)} (n)}\Big)_{0 \le n \le C} .\]
The argument is then the same as in the proof of Proposition ~\ref{prop:lyap}.

\vspace{1mm}
(iii) We only give a sketch of the proof, which  is the same as  in ~\cite{Antunes:08}. It is first proved that the infinitesimal generator $ \Omega ^N$ of the Markov jump process given by ~\eqref{eq:closed} converges, as $N, M\to \infty$ with $M=M(N)$ and $M/N \to \lambda$, to the degenerate generator given by: $\Omega f (y) =  y L_y \nabla f(y) $. This convergence holds for $C^2$ functions $f$ on $\Y$, and is uniform in $f$. It is then standard that any weak limit $\pi$ of the invariant distributions $(\pi ^N)$  of generators $\Omega ^N$,  solves $\pi \Omega =0$.

The Lyapunov function $g$ helps then proving that the Dirac mass at  the  unique equilibrium point  $\nu _{\rho _{\lambda}}$  of the dynamical system ~\eqref{eq:closed_dyn} is the only  invariant probability measure for generator $\Omega$. Indeed, $\pi \Omega =0$ implies in particular that $\pi \Omega g = \int _{\Y} yL_y \nabla g(y)  \, \pi (dy)=0$. The integrand being non-positive, and zero only at $\nu _{\rho _{\lambda}}$, $\pi$ needs be supported by this single state. 

One concludes that the Dirac mass at $\nu _{\rho _{\lambda}}$ is the only weak limit of the sequence $(\pi ^N)$, so by  compactness of  $\mathcal P(\Y)$, it is the limit of $(\pi ^N)$.

\end{proof}

\begin{remark}
Note that the above $g$ could be replaced by $h(y)=h(y| \nu _{\rho}) $ for any fixed $\rho$ since, by Proposition ~\ref{prop:prox}, $g$ and $h$ only differ by the constant $h(\nu _{\rho _{\lambda}}| \nu _{\rho}) $.

\end{remark}

\subsection{An open version} \label{subsec:open}
The next model is analogous to the previous one, in the sense that  customers are instantly directed to available nodes. But now,  new customers  enter the network, are served at some node and then leave the network. Only one class of customers,  requiring one unit of capacity, is considered (see Remark~\ref{remark:multiclass} below for possible extension to several classes). Customers enter the system, still consisting of $N$ nodes, according to some Poisson process with intensity $\lambda N$. Each customer is instantaneously directed, if possible,  toward  some node chosen uniformly among those having one free unit of capacity, that he then occupies during an exponentially distributed time with mean $1$. If no such node exists, the customer is definitively rejected.

We denote by $Y^N(t)$ ($t\ge 0$) the empirical distribution of the nodes at time $t$, defined as in ~\eqref{eq:Y}. Its state space is the finite subset  $\Y ^N =   \{y=(y_n,0 \le n\le C) \in \Y: Ny_n \in \N \text{ for } n= 0, \dots , C\}$ of  $\Y= \{y=(y_n,0 \le n\le C) \in [0, + \infty [ ^{C+1}:  \sum _ny_n=1 \}$.

 Notice that the total number of customers present in the system: $N \sum _n nY^N_n= N[I, Y^N]$  is simply an $M/M/CN/CN$ queue with arrival rate $\lambda N$ and service rate $1$.  Considering  $[I, Y^N]$ as $N$ grows to infinity is equivalent to  Kelly scaling for the $M/M/C/C$ queue. When $\lambda \ge C$ the situation is simple: In the limit, the renormalized  queue gets close to some deterministic trajectory that is constant  equal to $C$ after some time. This means that after some time, for large $N$, the system is saturated. 

It will be assumed that $0 <\lambda <C$, so as to maintain the system in a no-rejection regime. Indeed this case corresponds to the subcritical  regime of Kelly's asymptotic: $[I, Y^N]$  converges to some process  with values in $]0,C[$ for $t>0$ (having limit $\lambda $  at infinity) so that in the limit, no rejection occurs.\\

 $(Y^N(t), t \ge 0)$ is a family of irreducible Markov processes on $\Y$ with the following transitions, respectively corresponding to some arrival or departure at some node in state $n \in \X$:
\[ y \longrightarrow \left \{ \begin{array}{l}  y+\frac{1}{N} (e_{n+1}-e_n) \\ \\
 y+\frac{1}{N} (e_{n-1}-e_n)
\end{array} \right.
\text{at rate}  \begin{array}{ll}
 \lambda   N \frac{y_n}{1-y_C} \mathds 1 _{ y_C<1} & (0 \le n \le C-1) \\
\\ N n y_n & (0 \le n \le C)

\end{array} \]
Since each $Y^N$ actually evolves in a finite subset of $\Y$, these processes are ergodic; their invariant distributions will be denoted  $\pi ^N$. 

 Apart from  the particular case $C=1$, in which coordinate  $Y_1^N$  of $Y^N$  (that is the proportion of occupied nodes) is  itself a renormalized $M/M/N/N$ queue, $Y^N$ is non-reversible and its invariant distribution is not explicitly known.

The process $ Y^N(t)$ converges in distribution, on any finite time interval $[0,T]$,  to  the solution of the same differential system  ~\eqref{eq:closed_dyn} as in the  closed case, here considered on the enlarged
space $\Y \setminus \{\delta _C\}$, instead of $\Y_{\lambda}$. 
This holds provided that $Y^N(0)$ converges in distribution to  some $y(0)$ such that $y_C(0) <1$. The assumption that $\lambda <C$ is crucial here,   ensuring that the condition  $y_C <1$ is preserved  in time: Indeed the last equation in the above differential system writes $y'_C(t)= \lambda \frac{y_{C-1}(t)}{1-y_C(t)}  -C  y_C(t) \le \lambda -C   y_C(t)<0$  for $  y_C(t)$  close to $1$, so that $y_C$ cannot reach the value $1$, guaranteeing existence and unicity of  a solution to ~\eqref{eq:closed_dyn} defined for all positive times. 
\\

This system can  again be read as ~\eqref{eq:S} where $L_y$ is the infinitesimal generator of an $M/M/C/C$ queue with arrival rate $ \frac{\lambda }{1-y_C}  $ and service rate $1$. Equilibrium points are thus characterized by ~\eqref{eq:fix}, where $\rho (y)=  \frac{\lambda }{1-y_C}$, or  by $y=\nu _{\rho}$ with $\rho= \frac{\lambda}{1-B(\rho)} $. Then as in the previous model, using the relation $[ I,\nu _{\rho}]=\rho (1-B(\rho)) $ together with Proposition ~\ref{prop:prox},  $\nu _{\rho _{\lambda}}$ is the unique equilibrium point for ~\eqref{eq:S} on $\Y _{\lambda}$, where $\rho _{\lambda}$ is characterized by $[I, \nu _{\rho _{\lambda}}]=\lambda$. 

\begin{prop}
As $N$ goes to infinity, the invariant distribution $\pi ^N$ of process $Y^N$ converges to the Dirac mass at $\nu _{\rho _{\lambda}}$.

\end{prop}

\begin{proof}
The Lyapunov function $g(y)=h(y| \nu _{\rho _{\lambda}}) $ for the previous closed system is no longer a Lyapunov function for the present model, though the differential system is formally the same. Indeed the state space is here larger, consisting of $\Y \setminus \{\delta _C\}$ instead of $\Y_{\lambda}$, and the Lyapunov property of $g$ relied on invariance of $[I,y]$  along the dynamical system ~\eqref{eq:closed_dyn} restricted to $\Y_{\lambda}$: This is no longer valid on $\Y \setminus \{\delta _C\}$.

Nevertheless one can restrict to the set $\Y_{\lambda}$ once the following is noticed: The quantity $l(y)=([I,y]- \lambda)^2$ decreases along the flow. This results from relation ~\eqref{eq:mean} in the proof of Proposition ~\ref{prop:closed}, that remains valid except for the last equality. Here
\[ yL_y \nabla f(y)=  \lambda - [I,y]  \quad \text{ where } \quad  f(y) = [I,y] , \] 
so  that $ yL_y \nabla l(y)=  -2(\lambda -  [I,y] )^2 \le 0.$  

Arguing as in the proof of (iii) of Proposition ~\ref{prop:closed}  gives that any weak limit of the sequence $(\pi ^N)$ is supported by the set $\{y \in \Y: l(y)=0 \}$, that is $\Y_{\lambda }$. Then using  the Lyapunov function $g(y)=h(y| \nu _{\rho _{\lambda}}) $ for the system restricted to $\Y_{\lambda }$ shows as previously  that  the Dirac mass at $\nu _{\rho _{\lambda}}$ is the only weak limit for the  sequence $(\pi ^N)$, which proves convergence by compactness of $\mathcal P (\Y)$.

\end{proof}

\begin{remark} \label{remark:multiclass}
Extension to  several classes  of customers here preserves irreducibility but raises the question: Under what  condition, analogous to $\lambda<C$, should the previous results generalize?

Consider $K$ classes of customers with arrival rates $\lambda _k N$,  service rates $\mu _k$ and  capacity requirements  $A_k$ ($A_k \ge 1$) for $k=1, \dots ,K$ (customers  of class $k$  being  here again, if possible,   directed toward  some node chosen uniformly among those having $A_k$ free units of capacity, and being rejected otherwise).

Now the total number of customers is  no longer an $M/M/CN/CN$ queue: Indeed when some $A_k$ is larger than $1$, some class $k$ customer can be rejected though there are $A_k$  free units of capacity in the system, because no such volume of capacity is available at a single node.  The condition  $\sum _k A_k\frac{\lambda _k}{\mu _k} <C$  (ensuring that the $M/M/CN/CN$ queue with parameters $\lambda _k$'s, $\mu _k$'s, $A_k$'s  and $C$ is subcritical) is thus irrelevant. (As an example, consider  two classes of customers respectively requiring $A_1=3$ and $A_2=2$ units of capacity, while $C=3$. The process  is exactly the same as if  parameters were $A_1=A_2=C=1$, since each node can be occupied by at most one customer).

 The right  condition  should be $ (\lambda _k/\mu _k)_k \in \{ [I,y], y \in \stackrel{\circ} {\Y} \}$, which is equivalent to  existence of  a $ \rho $ such that $[I, \nu _{\rho }] = (\lambda _k/\mu _k)_k$, by Proposition ~\ref{prop:prox}. This condition says that  $ \lceil N \lambda _k/\mu _k \rceil$ customers of each class $k$ can be simultaneously accomodated for $N$ sufficiently large. 
But contrary to the one class case, this condition does not ensure that the  system  stays in a non blocking regime in the limit. For example in the very particular case when $A_1= \dots =  A_k =C=1$, the empirical  process   (simply consisting of the densities of customers of the different classes)   is itself a renormalized  $M/M/N/N$ queue. The previous condition writes $\sum _k \lambda _k/\mu _k <1$, which  is the subcritical regime of Kelly. The limiting dynamics is known (see  for example ~\cite{Fricker:03}) and for some values of the parameters, the trajectories  can spend  some non negligible time on  the blocking region $y_0=0$.       

\end{remark}

\end{document}